\documentclass[12pt,a4paper]{amsart}
\usepackage{amsfonts}
\usepackage{amsthm}
\usepackage{amsmath}
\usepackage{amscd}
\usepackage{amssymb}
\usepackage[latin2]{inputenc}
\usepackage{t1enc}
\usepackage[mathscr]{eucal}
\usepackage{indentfirst}
\usepackage{graphicx}
\usepackage{graphics}
\usepackage{pict2e}
\numberwithin{equation}{section}
\usepackage[margin=2.9cm]{geometry}
\usepackage{epstopdf}
\usepackage{mathrsfs}

\theoremstyle{plain}
\newtheorem{Th}{Theorem}[section]
\newtheorem{Lemma}[Th]{Lemma}

\newtheorem{Prop}[Th]{Proposition}
\newtheorem{Claim}[Th]{Claim}

 \theoremstyle{definition}
\newtheorem{Def}[Th]{Definition}

\newtheorem{?}[Th]{Problem}

\newcommand{\im}{\operatorname{im}}

\author{Antony T.H. Fung}
\date{\today}
\title{Every Jordan curve inscribes uncountably many rhombi}

\begin{document}

\begin{abstract}
We prove that every Jordan curve in $\mathbb{R}^2$ inscribes uncountably many rhombi. No regularity condition is assumed on the Jordan curve.
\end{abstract}

\maketitle

\section{Introduction}

The inscribed square problem is a famous open problem due to Otto Toeplitz in 1911 \cite{toeplitz}, which asks whether all Jordan curves in $\mathbb{R}^2$ inscribe a square. By ``inscribe'', it means that all four vertices lie on the curve. There are lots of results for the case when the Jordan curve is ``nice'' \cite{greene} \cite{hugelmeyer} \cite{matschke} \cite{tao}. However, other than Vaughan's result stating that every Jordan curve in $\mathbb{R}^2$ inscribes a rectangle (see \cite{meyerson}), little progress has been done on the general case.\\
\\
One of the earliest results was Arnold Emch's work \cite{emch}. In 1916, he proved that all piecewise analytic Jordan curves in $\mathbb{R}^2$ with only finitely many ``bad'' points inscribe a square. Emch's approach was that given a direction $\tau$, construct the set of medians $M_\tau$, defined as (roughly speaking) the set of midpoints of pairs of points on the curve that lie on the same line going in direction $\tau$. When the Jordan curve is nice enough, $M_\tau$ is a nice path. Now consider an orthogonal direction $\sigma$ and similarly construct $M_\sigma$. The paths $M_\tau$ and $M_\sigma$ must intersect, and their intersections correspond to quadrilaterals with diagonals perpendicularly bisecting each other, i.e. rhombi. Then by rotating and using the intermediate value theorem, he showed that at some direction, one of those rhombi is a square.\\
\\
In this paper, we will follow a similar approach to show the existence of rhombi. However, without assuming analyticity, $M_\tau$ may not necessarily be a path. To get around this, we will define a new class of object called a \textit{pseudopath} that has properties similar to a path. Then we will show that $M_\tau$ is a pseudopath, and use those properties of pseudopath to proceed the argument.\\
\\
In this paper, we prove the following:
\begin{Th} \label{main} Let $\gamma:S^1\rightarrow\mathbb{R}^2$ be a Jordan curve. Then there exists an open interval of angles such that there exists inscribed rhombi of all these angles. Furthermore, if $\gamma$ does not contain a special corner, then there exists inscribed rhombi of all angles.
\end{Th}

In this section, we will define what is a rhombus of an angle and what is a special corner. In this paper, an \textit{inscribed rhombus} is a set of 4 distinct points in $im(\gamma)$ such that those 4 points are the vertices of a rhombus in $\mathbb{R}^2$.\\
\\
A corollary of Theorem \ref{main} is that every Jordan curve in $\mathbb{R}^2$ inscribes uncountably many rhombi.

\begin{Def}[line of angle $\theta$]
A \textit{line of angle $\theta$} means a line in the form $x\sin\theta-y\cos\theta=$constant. i.e. a line making angle $\theta$ with the $x$-axis in an anti-clockwise manner.
\end{Def}

\begin{Def}[special corner of angle $\theta$]
Let $\gamma:S^1\rightarrow\mathbb{R}^2$ be a Jordan curve. A \textit{special corner of $\gamma$ of angle $\theta$} is a point $p\in im(\gamma)$ such that both the line of angle $\theta$ through $p$ and the line of angle $\theta+\dfrac{\pi}{2}$ through $p$ only intersect $im(\gamma)$ at $p$.
\end{Def}

\begin{center}
\includegraphics[width=0.5\textwidth]{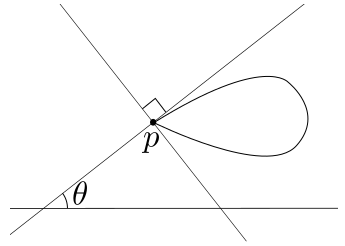}\\
\textbf{Figure 1:} $p$ is a special corner of angle $\theta$
\end{center}

\begin{Def}[special corner]
Let $\gamma:S^1\rightarrow\mathbb{R}^2$ be a Jordan curve. A point $p$ is a \textit{special corner} of $\gamma$ if it is a special corner of $\gamma$ of at least one angle $\theta$.
\end{Def}

Intuitively, a special corner is a point on $im(\gamma)$ such that translating to the origin and rotating, the entire curve lies within a single quadrant.

\begin{Def}[rhombus of angle $\theta$]
We say that a rhombus in $\mathbb{R}^2$ is a \textit{rhombus of angle $\theta$} if the two diagonals are lines with angles $\theta$ and $\theta+\dfrac{\pi}{2}$ respectively.
\end{Def}

Clearly, ``rhombus of angle $\theta$'' and ``rhombus of angle $\theta+\frac{\pi}{2}$'' are the same concept.\\
\\
Now we can state Theorem \ref{main} again in a more precise manner. We divide it into three separate statements.

\begin{Prop} \label{0 corner} Let $\gamma:S^1\rightarrow\mathbb{R}^2$ be a Jordan curve and $\theta\in\mathbb{R}$. If there is no special corner of angle $\theta$, then there exists an inscribed rhombus of angle $\theta$.
\end{Prop}

In particular, this implies that Jordan curves with no special corners have inscribed rhombi of all angles, and hence satisfying Theorem \ref{main}.

\begin{Prop} \label{1 corner} Let $\gamma:S^1\rightarrow\mathbb{R}^2$ be a Jordan curve with exactly one special corner. Then $\exists\theta_0$ and $\epsilon>0$ such that $\forall\theta\in(\theta_0-\epsilon,\theta_0+\epsilon)$, there is no special corner of angle $\theta$.
\end{Prop}

Together with Proposition \ref{0 corner}, this implies that Jordan curves with exactly one special corner satisfy Theorem \ref{main}.

\begin{Prop} \label{2 corner} Let $\gamma:S^1\rightarrow\mathbb{R}^2$ be a Jordan curve with at least two special corners. Let $p$ and $q$ be distinct special corners of $\gamma$, and suppose that the line passing through $p$ and $q$ is a line of angle $\theta_0$ for some $\theta_0$. Then $\exists\epsilon>0$ such that $\forall\theta\in(\theta_0-\epsilon,\theta_0+\epsilon)$, there exists an inscribed rhombus of angle $\theta$.
\end{Prop}

This implies that Jordan curves with at least two special corners satisfy Theorem \ref{main}.\\
\\
Together, Propositions \ref{0 corner}, \ref{1 corner}, \ref{2 corner} imply Theorem \ref{main}.\\
\\
We will quickly prove Proposition \ref{1 corner} first. Then in Section 2 we will develop some machinery in point-set topology to prove Proposition \ref{0 corner}. In Section 3 we will prove Proposition \ref{0 corner}. In Section 4 we will refine the arguments used in Section 3 to prove Proposition \ref{2 corner}, and hence completing the proof of Theorem \ref{main}.

\begin{proof}[Proof of Proposition \ref{1 corner}]
Suppose $p$ is the unique special corner of a Jordan curve $\gamma:S^1\rightarrow\mathbb{R}^2$. Let $q_1,q_2$ be two other points in $im(\gamma)$ such that $p,q_1,q_2$ are not collinear. Let $q$ be the mid-point of $q_1$ and $q_2$. For each $\theta$, let $l_\theta$ be the line of angle $\theta$ through $p$. Let $\theta_0$ be such that $l_{\theta_0}$ passes through $q$. Let $\epsilon=min(\angle q_2pq,\angle qpq_1)$. Then $\forall\theta\in(\theta_0-\epsilon,\theta_0+\epsilon)$, $q_1$ and $q_2$ are on opposite sides of the $l_\theta$. Since $im(\gamma)\backslash\{p\}$ is path-connected, $im(\gamma)$ must intersect $l_\theta$ at a point other than $p$, and hence $p$ is not a special corner of angle $\theta$. Since $p$ is the unique special corner of $\gamma$, there is no special corners of angle $\theta$.

\begin{center}
\includegraphics[width=0.5\textwidth]{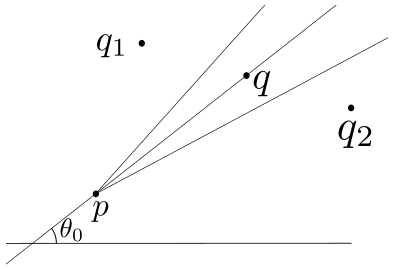}\\
\textbf{Figure 2:} Proof of Proposition \ref{1 corner}
\end{center}

\end{proof}

\section{Some point-set topology}

We begin with our key concept.

\begin{Def}[pseudopath]
Let $X$ be a topological space, and $p,q\in X$. A pseudopath between $p$ and $q$ is a compact set $C\subseteq X$ such that $p,q\in C$ and $\forall$ open set $U$ containing $C$, $\exists$ a path $\gamma$ from $p$ to $q$ satisfying $im(\gamma)\subseteq U$.
\end{Def}

Intuitively, a pseudopath is a compact set that is arbitrarily close to containing a path. We note some properties of pseudopaths.\\
\\
First of all, clearly the image of a path is a pseudopath. So pseudopath is a generalization of path. Another straightforward property of pseudopath is that the image of a pseudopath is a pseudopath:

\begin{Lemma} \label{image of pseudopath} Let $f:X\rightarrow Y$ be a continuous function between topological spaces. If $C$ be a pseudopath between the points $p,q\in X$, then $f(C)$ is a pseudopath between $f(p),f(q)$.
\end{Lemma}

\begin{proof}
Since C is compact, $f(C)$ is also compact. Since $p,q\in C$, $f(p)$ and $f(q)$ must be in $f(C)$. Suppose $U$ is open in $Y$ and contains $f(C)$. Then $f^{-1}(U)$ is open in $X$ and contains $C$, and hence contains the image of a path $\gamma$ from $p$ to $q$. Then $f\circ\gamma$ is our desired path from $f(p)$ to $f(q)$ with image lying within $U$.
\end{proof}

Now we prove the most crucial property of pseudopath in this paper.

\begin{Lemma} \label{exists pseudopath} Let $A,B,C,D$ be 4 distinct points on $\partial\mathbb{D}^2=S^1$ labeled in that order (i.e. $A,C$ lie in different path components in $\partial\mathbb{D}^2\backslash\{B,D\}$). Let $K$ be a compact set in $\mathbb{D}^2$ that intersects $\partial\mathbb{D}^2$ at $B$ and $D$, and only at $B$ and $D$. If $A,C$ lie in different path components in $\mathbb{D}^2\backslash K$, then $K$ is a pseudopath in $\mathbb{D}^2$ between $B$ and $D$.
\end{Lemma}

\begin{center}
\includegraphics[width=0.5\textwidth]{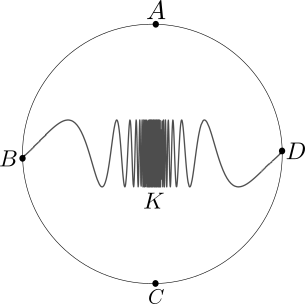}\\
\textbf{Figure 3:} An example of Lemma \ref{exists pseudopath}
\end{center}

Note that $K$ does not necessarily contain the image of a path from $B$ to $D$. For example, if we consider the example given in Figure 3 where $K$ is a set that looks like the graph of $y=\sin(\frac{1}{x})$, $x\neq 0$, with the hole at $x=0$ being filled. It does not contain the image of a path going from $B$ to $D$. That is precisely the reason why we have to deal with pseudopaths instead of just paths.\\
\\
The key to proving Lemma \ref{exists pseudopath} is Alexander duality, which requires local contractibility. Regarding the comment above on the example given in Figure 3, the graph of $y=\sin(\frac{1}{x})$ is not locally contractible at $0$, which prevents us from using Alexander duality directly. Our strategy of proving Lemma \ref{exists pseudopath} is that we first construct a locally contractible set, and then we apply Alexander duality.\\
\\
Now we prove Lemma \ref{exists pseudopath}.

\begin{proof}
We want to show that $K$ is a pseudopath in $\mathbb{D}^2$ between $B$ and $D$. Let $U$ be an arbitrary open set in $\mathbb{D}^2$ containing $K$. As $U$ is arbitrary, it suffices to show that inside $U$, there exists a path going from $B$ to $D$.\\
\\
Embed $\mathbb{D}^2$ in $\mathbb{R}^2$ in a standard way and give $\mathbb{D}^2$ the Euclidean metric inherited from $\mathbb{R}^2$. Now we can treat $\mathbb{D}^2$ as a metric space. Consider small open balls around $B$ and $D$ that are contained in $U$. Call them $B_{\epsilon_B}(B)$ and $B_{\epsilon_D}(D)$. Choose $\epsilon_B$ and $\epsilon_D$ to be small enough such that the closed balls $\overline{B_{\epsilon_B}}(B)$ and $\overline{B_{\epsilon_D}}(D)$ are contained in $U\backslash\{A,C\}$. Around each point in $k\in K\backslash\{B,D\}$, consider an open ball $B_{\epsilon_k}(k)$ contained within $U$. Choose $\epsilon_k$ to be small enough such that the closed ball $\overline{B_{\epsilon_k}}(k)$ is contained in $U\backslash\partial\mathbb{D}^2$.\\
\\
The open balls $B_{\epsilon_B}(B)$, $B_{\epsilon_D}(D)$, and those $B_{\epsilon_k}(k)$'s together form an open cover of $K$. By compactness, there exists a subcover that contains only finitely many $B_{\epsilon_k}(k)$'s. Let $X$ be the union of $\overline{B_{\epsilon_B}}(B)$ and $\overline{B_{\epsilon_D}}(D)$ and all those $\overline{B_{\epsilon_k}}(k)$'s for each $B_{\epsilon_k}(k)$ in the subcover. Since $X$ is a union of finitely many closed balls, it must be locally contractible (which is required for using Alexander duality), and $\mathbb{D}^2\backslash X$ has finitely many path components. Let $n$ be the number of path components in $\mathbb{D}^2\backslash X$. Also, by construction, $K\subseteq X\subseteq U\backslash\{A,C\}$, and $\partial\mathbb{D}^2\backslash X$ has exactly 2 path components, one containing $A$ and one containing $C$ (because $\overline{B_{\epsilon_B}}(B)$ and $\overline{B_{\epsilon_D}}(D)$ are the only closed balls used that can intersect $\partial\mathbb{D}^2$).\\
\\
Let $I$ be an arc of a very big circle in $\mathbb{R}^2$ that connects $B$ and $D$. Choose $I$ be the arc that only intersects $\mathbb{D}^2$ at $B$ and $D$.

\begin{center}
\includegraphics[width=0.5\textwidth]{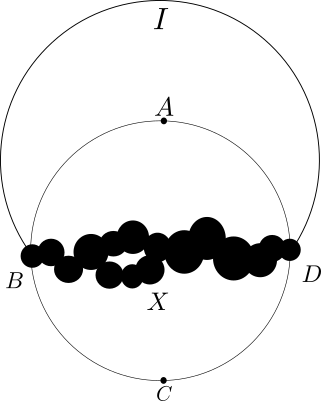}\\
\textbf{Figure 4:} $X$ and $I$. $X$ is a finite union of closed balls that covers $K$.
\end{center}

Here, we view $S^2$ as the one point compactification of $\mathbb{R}^2$.\\
\\
Inside $S^2\backslash X$, $A$ and $C$ are in the same path component, and the path components of $A$ and $C$ in $\mathbb{D}^2\backslash X$ are the only ones that merged when embedded in $S^2\backslash X$ because there are only 2 path components in $\partial\mathbb{D}^2$. Hence, $H_0(S^2\backslash X)=\mathbb{Z}^{n-1}$. Note that $\mathbb{Z}$ coefficients are implicit. By Alexander duality and universal coefficient theorem, $H_1(X)=\mathbb{Z}^{n-2}$.\\
\\
Now we consider $X\cup I$. In $S^2\backslash (int(\mathbb{D}^2)\cup I)$, $I$ separates $A$ and $C$ into different path components. Note that $S^2\backslash (X\cup I)$ is the union of the two closed subspaces $S^2\backslash (int(\mathbb{D}^2)\cup I)$ and $\mathbb{D}^2\backslash X$, and the two closed subspaces intersect at $\partial\mathbb{D}^2\backslash X$. $\partial\mathbb{D}^2\backslash X$ has only 2 path components, the one with $A$ and the one with $C$. Hence, all the path components in $\mathbb{D}^2\backslash X$ except for the two containing $A$ or $C$ remain being separate path components in $S^2\backslash (X\cup I)$. Those two path components containing $A$ or $C$ also remain separate because they are separate in $S^2\backslash (int(\mathbb{D}^2)\cup I)$. Hence, $H_0(S^2\backslash (X\cup I))=\mathbb{Z}^n$. By Alexander duality and universal coefficient theorem, $H_1(X\cup I)=\mathbb{Z}^{n-1}$. Therefore, $H_1(X)\not\cong H_1(X\cup I)$.\\
\\
Now we consider the Mayer-Vietoris sequence\\
$\cdots\rightarrow H_1(\{B,D\})\rightarrow H_1(X)\oplus H_1(I)\rightarrow H_1(X\cup I)\rightarrow H_0(\{B,D\})\rightarrow H_0(X)\oplus H_0(I)\rightarrow\cdots$\\
\\
Note that $H_1(\{B,D\})$ and $H_1(I)$ are $0$, and $H_1(X)\not\cong H_1(X\cup I)$. So, the map $H_1(X\cup I)\rightarrow H_0(\{B,D\})$ is not the zero map, and hence the map $H_0(\{B,D\})\rightarrow H_0(X)\oplus H_0(I)$ cannot be injective. Therefore, under the map $H_0(\{B,D\})\rightarrow H_0(I)$ induced by inclusion, $[B]$ and $[D]$ must be mapped to the same path component in $X$, and hence there exists a path within $X$ that goes from $B$ to $D$.\\
\\
As $X\subseteq U$, the proof of Lemma \ref{exists pseudopath} is completed.
\end{proof}

We prove one more property of pseudopaths before we go back to proving our theorem about Jordan curves.

\begin{Lemma} \label{pseudopaths intersect} Let $A,B,C,D$ be 4 distinct points on $\partial\mathbb{D}^2=S^1$ labeled in that order (i.e. $A,C$ lie in different path components in $\partial\mathbb{D}^2\backslash\{B,D\}$). Let $K$ be a pseudopath in $\mathbb{D}^2$ between $A,C$, and $L$ be a pseudopath in $\mathbb{D}^2$ between $B,D$. Then $K\cap L\neq\varnothing$.
\end{Lemma}

\begin{proof}
We proceed by contradiction. Assume that $K\cap L=\varnothing$. As $K$ and $L$ are compact sets in a Hausdorff space, there exists open sets $U$,$V$ such that $K\subseteq U$ and $L\subseteq V$ and $U\cap V=\varnothing$. So, by the definition of pseudopath, there exists a path from $A$ to $C$ and a path from $B$ to $D$ that do not intersect each other, which is impossible.
\end{proof}

\section{Proof of Proposition \ref{0 corner}}

This section is entirely devoted to the proof of Proposition \ref{0 corner}.\\
\\
We fix a Jordan curve $\gamma:S^1\rightarrow\mathbb{R}^2$, and suppose $\theta$ is an angle with no special corners.\\
\\
Note that we will not use the fact that $\theta$ is an angle with no special corners until the end of this section. Claim \ref{median is pseudopath} and everything before holds for all $\theta$, not just for those without special corners.\\
\\
Consider the function $(x,y)\mapsto x\sin\theta-y\cos\theta$ on $im(\gamma)$. As $im(\gamma)$ is compact, the function has a maximum and a minimum. Let $M_\theta$ be the maximum and $m_\theta$ be the minimum. Let $\mu_\theta:=\dfrac{M_\theta+m_\theta}{2}$.\\
\\
Note that the line $x\sin\theta-y\cos\theta=M_\theta$ can be parametrized as $\begin{cases}x=t\cos\theta+M_\theta\sin\theta \\ y=t\sin\theta-M_\theta\cos\theta\end{cases}$. Among all the points that are both in $\im(\gamma)$ and on the line $x\sin\theta-y\cos\theta=M_\theta$, let $MM_\theta$ be the point that attains the maximal $t$ under this parametrization ($MM_\theta$ exists by compactness). Similarly define $mM_\theta$ to be the one with minimal $t$.\\
\\
(Note that $MM_\theta$ and $mM_\theta$ may not necessarily be distinct)\\
\\
Similarly, define $Mm_\theta$ and $mm_\theta$ using the line $x\sin\theta-y\cos\theta=m_\theta$, and define $M\mu_\theta$ and $m\mu_\theta$ using the line $x\sin\theta-y\cos\theta=\mu_\theta$. Let $t^{min}_\theta$ be the $t$ value that realizes $m\mu_\theta$ under that parametrization, and let $t^{max}_\theta$ be the $t$ value that realizes $M\mu_\theta$ under that parametrization. Let $A_\theta$ be the mid-point of $Mm_\theta$ and $mm_\theta$, and $B_\theta$ be the mid-point of $MM_\theta$ and $mM_\theta$.

\begin{center}
\includegraphics[width=0.5\textwidth]{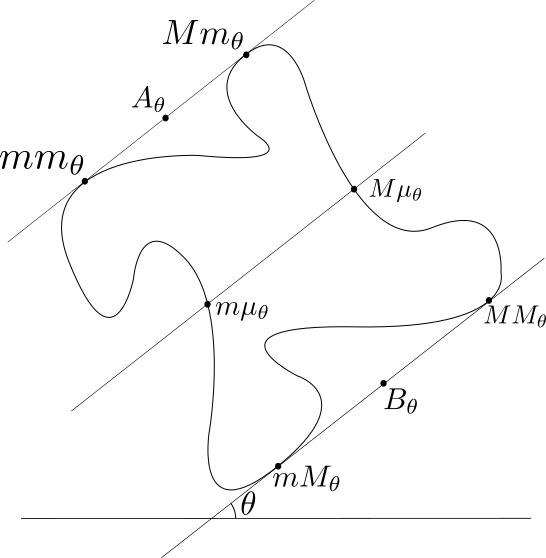}\\
\textbf{Figure 5:} Constructions of $mm_\theta$, $Mm_\theta$, $mM_\theta$, $MM_\theta$, $m\mu_\theta$, $M\mu_\theta$, $A_\theta$, $B_\theta$
\end{center}

Along the Jordan curve $\gamma$, there are two arcs from $mm_\theta$ to $mM_\theta$, one containing $m\mu_\theta$ and one doesn't. Let $\gamma_\theta$ be the arc containing $m\mu_\theta$, parametrized with $\gamma_\theta(0)=mm_\theta$ and $\gamma_\theta(1)=mM_\theta$, and $\gamma_\theta:[0,1]\rightarrow\mathbb{R}^2$ being injective. Let $\gamma^{o}_\theta$ be the arc not containing $m\mu_\theta$, parametrized with $\gamma^{o}_\theta(0)=mm_\theta$ and $\gamma^{o}_\theta(1)=mM_\theta$.\\
\\
Similarly, let $\Gamma_\theta$ be the arc along $\gamma$ from $Mm_\theta$ to $MM_\theta$ that passes through $M\mu_\theta$, parametrized with $\Gamma_\theta(0)=Mm_\theta$ and $\Gamma_\theta(1)=MM_\theta$, and $\Gamma_\theta:[0,1]\rightarrow\mathbb{R}^2$ being injective.\\
\\
Now we prove some properties of $\gamma_\theta$ and $\Gamma_\theta$.

\begin{Claim} \label{touch at end points only}
The following 4 statements are all true:
\begin{enumerate}
\item
$im(\gamma_\theta)$ and the line $x\sin\theta-y\cos\theta=m_\theta$ only intersect at $mm_\theta$

\item
$im(\gamma_\theta)$ and the line $x\sin\theta-y\cos\theta=M_\theta$ only intersect at $mM_\theta$

\item
$im(\Gamma_\theta)$ and the line $x\sin\theta-y\cos\theta=m_\theta$ only intersect at $Mm_\theta$

\item
$im(\Gamma_\theta)$ and the line $x\sin\theta-y\cos\theta=M_\theta$ only intersect at $MM_\theta$
\end{enumerate}
\end{Claim}

\begin{proof}
By symmetry, we only have to prove (1). We will prove it by contradiction.\\
\\
Suppose there is a point $pm_\theta\in im(\gamma_\theta)$ lying on the line $x\sin\theta-y\cos\theta=m_\theta$ with $pm_\theta\neq mm_\theta$. Both $pm_\theta$ and $m\mu_\theta$ are in $im(\gamma_\theta)$, so we can let $\gamma_\theta|_{[a,b]}$ to be a path either going from $pm_\theta$ to $m\mu_\theta$ or going from $m\mu_\theta$ to $pm_\theta$, for some $0<a<b<1$.\\
\\
Since $\gamma$ is a Jordan curve, it cannot self-intersect. So $im(\gamma^{o}_\theta)\cap im(\gamma_\theta|_{[a,b]})=\varnothing$. By the definition of $M_\theta$ and $m_\theta$, $im(\gamma^{o}_\theta)$ has to lie within the stripe $S_\theta:=\{(x,y) \mid m_\theta\leq x\sin\theta-y\cos\theta\leq M_\theta\}$. Also, by the definition of $m\mu_\theta$, $im(\gamma^{o}_\theta)$ cannot intersect the ray $L_\theta:=\{(t\cos\theta+\mu_\theta\sin\theta,t\sin\theta-\mu_\theta\cos\theta) \mid t\leq t^{min}_\theta\}$.\\
\\
So $im(\gamma^{o}_\theta)\subseteq S_\theta\backslash(im(\gamma_\theta|_{[a,b]})\cup L_\theta)$. However, in $S_\theta$, $im(\gamma_\theta|_{[a,b]})\cup L_\theta$ separates $mm_\theta$ and $mM_\theta$ into two different path components. So we have arrived a contradiction.

\begin{center}
\includegraphics[width=0.4\textwidth]{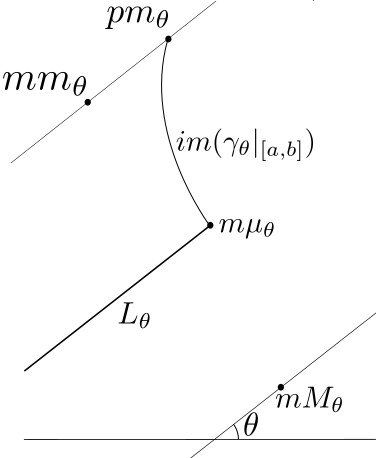}\\
\textbf{Figure 6:} Proof of Claim \ref{touch at end points only}
\end{center}

\end{proof}

\begin{Claim} \label{distinct in interior}
$im(\gamma_\theta|_{(0,1)})\cap im(\Gamma_\theta|_{(0,1)})=\varnothing$
(i.e. $\gamma_\theta$ and $\Gamma_\theta$ don't intersect, except possibly at the endpoints)
\end{Claim}

\begin{proof}
\noindent Case 1: If $MM_\theta\neq mM_\theta$ and $Mm_\theta\neq mm_\theta$,
then by Claim \ref{touch at end points only}, $MM_\theta,Mm_\theta\notin im{\gamma_\theta}$ and $mM_\theta,mm_\theta\notin im{\Gamma_\theta}$. So this is the only possible configuration:

\begin{center}
\includegraphics[width=0.4\textwidth]{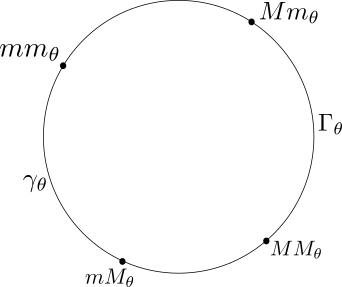}\\
\textbf{Figure 7:} Proof of Claim \ref{distinct in interior} Case 1 (up to homeomorphism)
\end{center}

\noindent Case 2: If $MM_\theta=mM_\theta$ and $Mm_\theta\neq mm_\theta$,
then by Claim \ref{touch at end points only}, $Mm_\theta\notin im{\gamma_\theta}$ and $mm_\theta\notin im{\Gamma_\theta}$. So this is the only possible configuration:

\begin{center}
\includegraphics[width=0.4\textwidth]{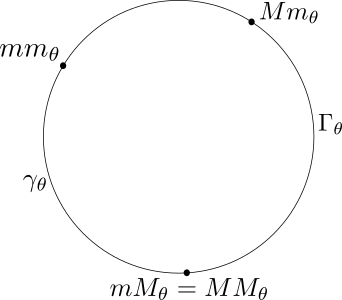}\\
\textbf{Figure 8:} Proof of Claim \ref{distinct in interior} Case 2 (up to homeomorphism)
\end{center}

\noindent Case 3: If $MM_\theta\neq mM_\theta$ and $Mm_\theta=mm_\theta$,
then the proof is similar to Case 2.\\
\\
Case 4: If $MM_\theta=mM_\theta$ and $Mm_\theta=mm_\theta$,
then what we need to prove is proving that $im(\gamma_\theta)\neq im(\Gamma_\theta)$. It suffices to show that $M\mu_\theta\notin im(\gamma_\theta)$. We will prove it by contradiction, in a similar fashion as the proof of Claim \ref{touch at end points only}.\\
\\
Suppose $M\mu_\theta\in im(\gamma_\theta)$. Then we let $\gamma_\theta|_{[a,b]}$ to be a path either going from $M\mu_\theta$ to $m\mu_\theta$ or going from $m\mu_\theta$ to $M\mu_\theta$, for some $0<a<b<1$. Let
\[L_\theta:=\{(t\cos\theta+\mu_\theta\sin\theta,t\sin\theta-\mu_\theta\cos\theta) \mid t\leq t^{min}_\theta\}\]
\[U_\theta:=\{(t\cos\theta+\mu_\theta\sin\theta,t\sin\theta-\mu_\theta\cos\theta) \mid t\geq t^{max}_\theta\}\]

Since $\gamma$ is a Jordan curve, $im(\gamma^{o}_\theta)\cap im(\gamma_\theta|_{[a,b]})=\varnothing$. By the definition of $M_\theta$ and $m_\theta$, $im(\gamma^{o}_\theta)$ has to lie within the stripe $S_\theta:=\{(x,y) \mid m_\theta\leq x\sin\theta-y\cos\theta\leq M_\theta\}$. Also, by the definition of $m\mu_\theta$ and $M\mu_\theta$, $im(\gamma^{o}_\theta)$ cannot intersect $L_\theta\cup U_\theta$.\\
\\
So, $im(\gamma^{o}_\theta)\subseteq S_\theta\backslash(im(\gamma_\theta|_{[a,b]})\cup L_\theta\cup U_\theta)$. However, in $S_\theta$, $im(\gamma_\theta|_{[a,b]})\cup L_\theta\cup U_\theta$ separates $mm_\theta$ and $mM_\theta$ into two different path components. So we have arrived a contradiction.

\begin{center}
\includegraphics[width=0.5\textwidth]{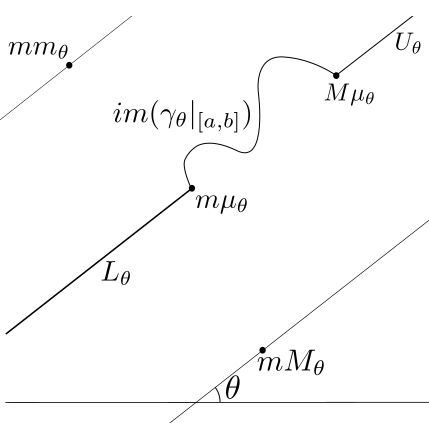}\\
\textbf{Figure 9:} Proof of Claim \ref{distinct in interior} Case 4
\end{center}

\end{proof}

We have finished proving Claim \ref{touch at end points only} and Claim \ref{distinct in interior}, which are properties of $\gamma_\theta$ and $\Gamma_\theta$. Now we define the \textit{median} of the angle $\theta$, denoted as $\mathcal{M}_\theta$:\\
\\
$\mathcal{M}_\theta:=\{\text{mid-point of } p_1 \text{ and } p_2 \mid p_1\in im(\gamma_\theta), \ p_2\in im(\Gamma_\theta), \ p_1 \text{ and } p_2 \text{ lie on the same}$\\
$\text{ line of angle } \theta\}$\\
\\
Remark: We call it ``median'' here because the construction is similar to Emch's construction of medians in his paper about inscribed rhombi in piecewise analytic Jordan curves \cite{emch}. The only difference is that analyticity is not assumed here.\\
\\
We will now prove that $\mathcal{M}_\theta$ is a pseudopath from $A_\theta$ to $B_\theta$. When we were defining the concept of pseudopaths, we worked inside a topological space. Therefore, we will first construct a topological space $Rec_\theta$ for us to work in.\\
\\
Consider $S_\theta:=\{(x,y) \mid m_\theta\leq x\sin\theta-y\cos\theta\leq M_\theta\}$. For all constructions so far, we do the same construction for every angle, not just $\theta$.\\
\\
Let $Rec_\theta:=S_\theta\cap S_{\theta+\frac{\pi}{2}}$. Clearly, $Rec_\theta=Rec_{\theta+\frac{\pi}{2}}$.

\begin{center}
\includegraphics[width=0.5\textwidth]{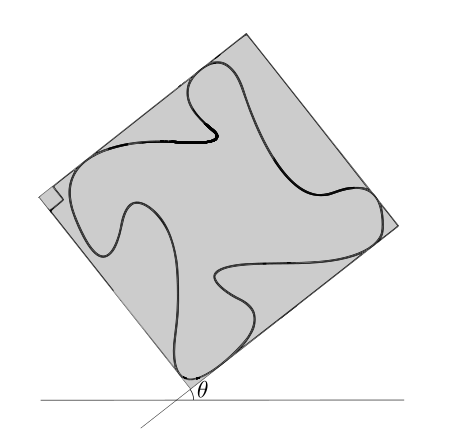}\\
\textbf{Figure 10:} $Rec_\theta$
\end{center}

\begin{Claim} \label{median is pseudopath}
In $Rec_\theta$, $\mathcal{M}_\theta$ is a pseudopath between $A_\theta$ and $B_\theta$.
\end{Claim}

\begin{proof}
Let $\pi_\theta:\mathbb{R}^2\rightarrow\mathbb{R}$ be the projection $(x,y)\rightarrow x\sin\theta-y\cos\theta$.\\
Let $f_\theta:[0,1]^2\rightarrow\mathbb{R}$ be the function $(r_1,r_2)\rightarrow\pi_\theta\circ\gamma_\theta(r_1)-\pi_\theta\circ\Gamma_\theta(r_2)$.\\
Let $g_\theta:[0,1]^2\rightarrow\mathbb{R}^2$ be the function $(r_1,r_2)\rightarrow\text{mid-point of }\gamma_\theta(r_1)\text{ and }\Gamma_\theta(r_2)$.\\
Clearly, $\pi_\theta,f_\theta,g_\theta$ are continuous as functions between topological spaces. Also, by definition, $\mathcal{M}_\theta=g_\theta(f_\theta^{-1}(0))$.\\
\\
Suppose $\alpha$ is a path in $[0,1]^2$ going from $(0,1)$ to $(1,0)$. Then $f_\theta\circ\alpha$ is a path in $\mathbb{R}$ going from a negative number to a positive number, hence must passes through $0$. So $im(\alpha)\cap f_\theta^{-1}(0)\neq\varnothing$.\\
\\
As $\alpha$ is arbitrary, $(0,1)$ and $(1,0)$ are in different path components in $[0,1]^2\backslash f_\theta^{-1}(0)$. By claim \ref{touch at end points only}, $f_\theta^{-1}(0)$ only intersects $\partial[0,1]^2$ at $(0,0)$ and $(1,1)$. Also, $f_\theta^{-1}(0)$ is compact in $[0,1]^2$ because $\{0\}$ is closed in $\mathbb{R}$ and $[0,1]^2$ itself is compact. So, by Lemma \ref{exists pseudopath}, $f_\theta^{-1}(0)$ is a pseudopath from $(0,0)$ to $(1,1)$. Then by Lemma \ref{image of pseudopath}, $\mathcal{M}_\theta=g_\theta(f_\theta^{-1}(0))$ is a pseudopath between $A_\theta$ and $B_\theta$.
\end{proof}

Now we will complete the proof of Proposition \ref{0 corner}, which is the aim of Section 3. When $\gamma$ has no special corners of angle $\theta$, $A_\theta$,$A_{\theta+\frac{\pi}{2}}$,$B_\theta$,$B_{\theta+\frac{\pi}{2}}$ are all distinct. By Claim \ref{median is pseudopath}, $\mathcal{M}_\theta$ is a pseudopath between $A_\theta$ and $B_\theta$, and $\mathcal{M}_{\theta+\frac{\pi}{2}}$ is a pseudopath between $A_{\theta+\frac{\pi}{2}}$ and $B_{\theta+\frac{\pi}{2}}$. By Lemma \ref{pseudopaths intersect}, $\mathcal{M}_\theta\cap \mathcal{M}_{\theta+\frac{\pi}{2}}\neq\varnothing$.\\
\\
By Claim \ref{touch at end points only} and Claim \ref{distinct in interior}, $\mathcal{M}_\theta$ only intersects $\partial Rec_\theta$ at $A_\theta$ and $B_\theta$, and $\mathcal{M}_{\theta+\frac{\pi}{2}}$ only intersects $\partial Rec_\theta$ at $A_\theta+\frac{\pi}{2}$ and $B_\theta+\frac{\pi}{2}$.\\
\\
Hence, $\mathcal{M}_\theta\cap \mathcal{M}_{\theta+\frac{\pi}{2}}\subseteq int(Rec_\theta)$. So, by Claim \ref{distinct in interior} and the definition of $\mathcal{M}_\theta$ and $\mathcal{M}_{\theta+\frac{\pi}{2}}$, intersection points of $\mathcal{M}_\theta$ and $\mathcal{M}_{\theta+\frac{\pi}{2}}$ correspond to inscribed quadrilaterals with diagonals being lines of angle $\theta$ and angle $\theta+\frac{\pi}{2}$, i.e. inscribed rhombi of angle $\theta$. Since $\mathcal{M}_\theta\cap \mathcal{M}_{\theta+\frac{\pi}{2}}\neq\varnothing$, there exists an inscribed rhombus of angle $\theta$.

\section{Proof of Proposition \ref{2 corner}}

This section is entirely devoted to the proof of Proposition \ref{2 corner}.\\
\\
Let $\gamma$ be a Jordan curve in $\mathbb{R}^2$ with at least two special corners. Let $p,q$ be distinct special corners of $\gamma$. Without loss of generality, using suitable rotations and translations, we can let $p$ be the origin of $\mathbb{R}^2$ and let $q$ be lying in the positive $x$-axis. What we need to prove is that $\exists\epsilon>0$ $\forall\theta\in(-\epsilon,\epsilon)$ $\exists$ an inscribed rhombus of angle $\theta$. By symmetry, we can replace $(-\epsilon,\epsilon)$ by $[0,\epsilon)$.\\
\\
Let $p$ be a special corner of angle $\theta_p$, where $\theta_p\in[0,\frac{\pi}{2})$. Actually, $\theta_p$ cannot be $0$ because $q$ is in $im(\gamma)$.\\
(Note that the choice of $\theta_p$ may not necessarily be unique)\\
\\
Similarly, let $q$ be a special corner of angle $\theta_q$, with $q\in(0,\frac{\pi}{2})$. Then $im(\gamma)$ is inside the following region, touching the boundary only at $p$ and $q$:

\begin{center}
\includegraphics[width=0.5\textwidth]{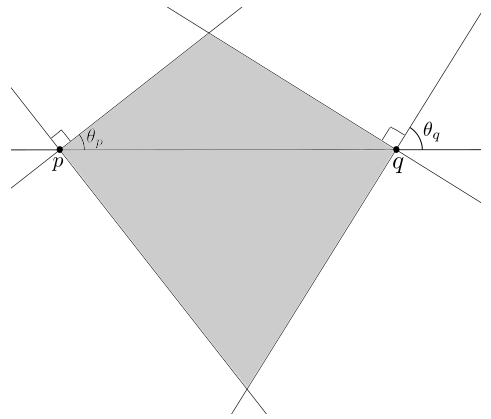}\\
\textbf{Figure 11:} A region where $im(\gamma)$ must lie within
\end{center}

It is impossible for all points on a Jordan curve to be collinear. So, $A_0$ and $B_0$ cannot both lie on the $x$-axis. Without loss of generality (by reflecting along the $x$-axis if necessary), let $A_0$ be lying strictly above the $x$-axis. Let $\epsilon_A$ be the angle $\angle qpA_0$. i.e. the line going through the origin and $A_0$ has angle $\epsilon_A\in(0,\theta_p)$. Let $\epsilon_B$ be the angle $\angle B_0 qp$. i.e. the line going through $B_0$ and $q$ has angle $\epsilon_B\in[0,\theta_q)$.\\
\\
Now we will split it into two cases. Case 1 is when $\epsilon_B>0$, i.e. parts of the Jordan curve lies below the $x$-axis. Case 2 is when $\epsilon_B=0$, i.e. the entire Jordan curve lies on or above the $x$-axis.\\
\\
Case 1: When $\epsilon_B>0$.\\
Let $\epsilon:=min(\epsilon_A,\epsilon_B)$. Then $\forall\theta\in[0,\epsilon)$, $A_{\theta+\frac{\pi}{2}}=p$ (because $\theta<\theta_p$) and $B_{\theta+\frac{\pi}{2}}=q$ (because $\theta<\theta_q$). Also, $A_{\theta}$ is strictly above the line $y=x\tan\theta$ (because $\theta<\epsilon_A$) and  $B_{\theta}$ is strictly below the line of angle $\theta$ passing through $q$ (because $\theta<\epsilon_B$). So, $A_{\theta+\frac{\pi}{2}}$, $B_{\theta+\frac{\pi}{2}}=q$, $A_{\theta}$, $B_{\theta}$ are all distinct, and the argument in Section 3 still goes through. Hence, there exists an inscribed rhombus of angle $\theta$.

\begin{center}
\includegraphics[width=0.7\textwidth]{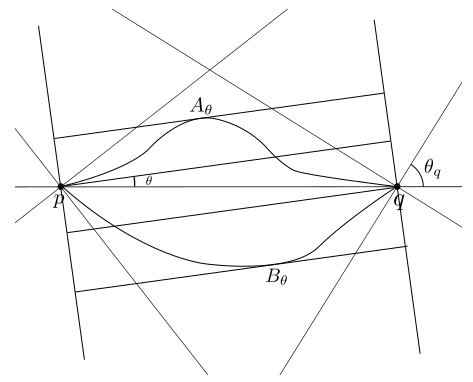}\\
\textbf{Figure 12:} When $B_\theta$ is strictly below the $x$-axis
\end{center}

Case 2: When $\epsilon_B=0$.\\
\\
Recall that $\Gamma_\frac{\pi}{2}:[0,1]\rightarrow\mathbb{R}^2$ is a path from $p=Mm_\frac{\pi}{2}$ to $q=MM_\frac{\pi}{2}$. $\Gamma_\frac{\pi}{2}$ is the path ``above'' and $\gamma_\frac{\pi}{2}$ is the path ``below''. (precise definition is in Section 3)\\
\\
Let $w$ be the $x$-coordinate of $q$ (the width of $Rec_0$), and $h$ be the $y$-coordinate of $A_0$ (the height of $Rec_0$).\\
Let $\epsilon_l:=min\{x\text{-coordinate of }\Gamma_\frac{\pi}{2}(r) \mid r\in[\Gamma_\frac{\pi}{2}^{-1}(Mm_{\epsilon_A}),1]\}$.\\
Let $\epsilon_r:=min\{w-(x\text{-coordinate of }\Gamma_\frac{\pi}{2}(r)) \mid r\in[0,\Gamma_\frac{\pi}{2}^{-1}(mm_0)]\}$.\\
\\
Note that $\epsilon_l$ and $\epsilon_r$ exist because of compactness.

\begin{center}
\includegraphics[width=0.5\textwidth]{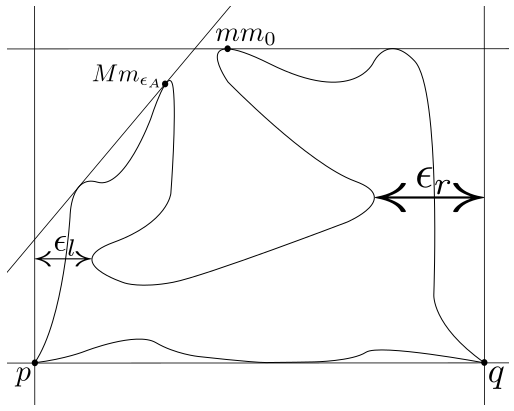}\\
\textbf{Figure 13:} Definition of $\epsilon_l$ and $\epsilon_r$
\end{center}

Let $\epsilon_y:=min\{y \mid (x,y)\in im(\Gamma_\frac{\pi}{2})\text{ and }\frac{\epsilon_l}{8}\leq x\leq w-\frac{\epsilon_r}{8}\}$. Note that $\epsilon_y>0$ because $im(\Gamma_\frac{\pi}{2})$ only intersects the $x$-axis at $p$ and $q$.

\begin{center}
\includegraphics[width=0.5\textwidth]{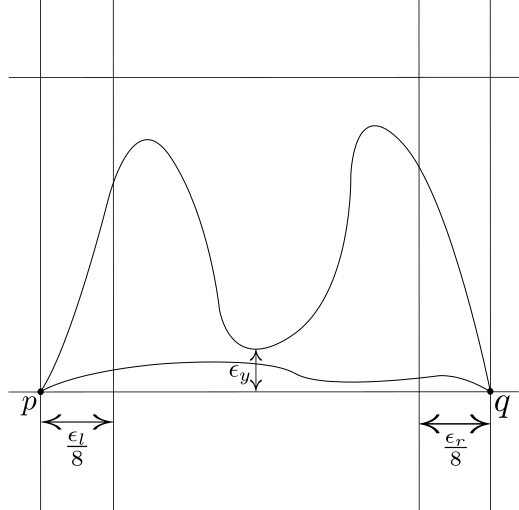}\\
\textbf{Figure 14:} Definition of $\epsilon_y$
\end{center}

We pick $\epsilon>0$ small enough such that
\[\tan\epsilon<\frac{\epsilon_l}{8h},\frac{\epsilon_r}{8h},\frac{\epsilon_y}{2w}\]
and
\[\epsilon<\epsilon_A,\theta_p,\theta_q\]

We want to show that $\forall\theta\in[0,\epsilon)$, there exists an inscribed rhombus of angle $\theta$. We first work on the $\theta=0$ case.\\
\\
The point $A_0$ is strictly above the $x$-axis. Also, $A_\frac{\pi}{2}=p=(0,0)$, $B_0=(\frac{w}{2},0)$, and $B_\frac{\pi}{2}=q=(w,0)$. Those 4 points are distinct, and hence the argument from Section 3 still goes through, and hence an inscribed rhombus of angle $0$ exists.\\
\\
We proceed with the $0<\theta<\epsilon$ case.\\
\\
Let $R$ be the region $\{(x,y) \mid \frac{\epsilon_l}{4}\leq x\leq w-\frac{\epsilon_r}{4}\text{ and }m_\theta\leq x\sin\theta-y\cos\theta\leq 0\}$
Note that $\partial R$ is a parallelogram with two sides being vertical and two sides being line segments of angle $\theta$.

\begin{center}
\includegraphics[width=0.5\textwidth]{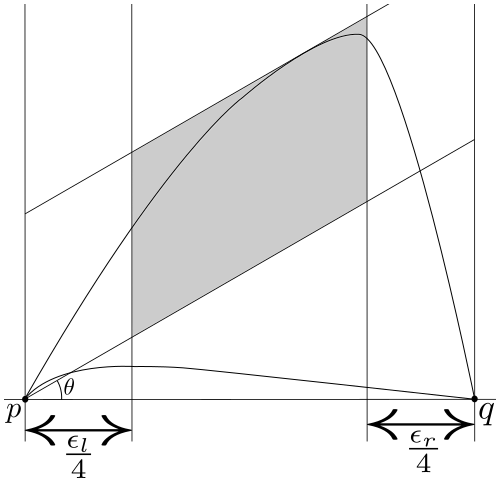}\\
\textbf{Figure 15:} The region $R$
\end{center}

We use a similar argument as the one used in Section 3. But instead of applying Lemma \ref{pseudopaths intersect} to $Rec_\theta$, we will apply it to $R$.\\
\\
\begin{Claim} \label{path order}
If $0\leq\theta_1<\theta_2<\frac{\pi}{2}$, then
$$\Gamma_\frac{\pi}{2}^{-1}(mm_{\theta_2})\leq\Gamma_\frac{\pi}{2}^{-1}(Mm_{\theta_2})\leq\Gamma_\frac{\pi}{2}^{-1}(mm_{\theta_1})\leq\Gamma_\frac{\pi}{2}^{-1}(Mm_{\theta_1})$$
\end{Claim}

\begin{proof}
We first prove that $\Gamma_\frac{\pi}{2}^{-1}(mm_{\theta_2})\leq\Gamma_\frac{\pi}{2}^{-1}(Mm_{\theta_2})$. We will prove it by contradiction.\\
\\
Suppose $\Gamma_\frac{\pi}{2}^{-1}(mm_{\theta_2})>\Gamma_\frac{\pi}{2}^{-1}(Mm_{\theta_2})$.\\
Let $T_{\theta_2}:=\{(x,y) \mid 0\leq x\leq w\text{ and }y\geq 0\text{ and }x\sin\theta_2-y\cos\theta_2\geq m_{\theta_2}\}$.\\
Then $\Gamma_\frac{\pi}{2}|_{[0,\Gamma_\frac{\pi}{2}^{-1}(Mm_{\theta_2})]}$ is a path in $T_{\theta_2}\cong\mathbb{D}^2$ going from $p$ to $Mm_{\theta_2}$, and $\Gamma_\frac{\pi}{2}|_{[\Gamma_\frac{\pi}{2}^{-1}(mm_{\theta_2}),1]}$ is a path in $T_{\theta_2}$ going from $mm_{\theta_2}$ to $q$. However, $\Gamma_\frac{\pi}{2}|_{[0,\Gamma_\frac{\pi}{2}^{-1}(Mm_{\theta_2})]}$ and $\Gamma_\frac{\pi}{2}|_{[\Gamma_\frac{\pi}{2}^{-1}(mm_{\theta_2}),1]}$ do not intersect, which is impossible.

\begin{center}
\includegraphics[width=0.4\textwidth]{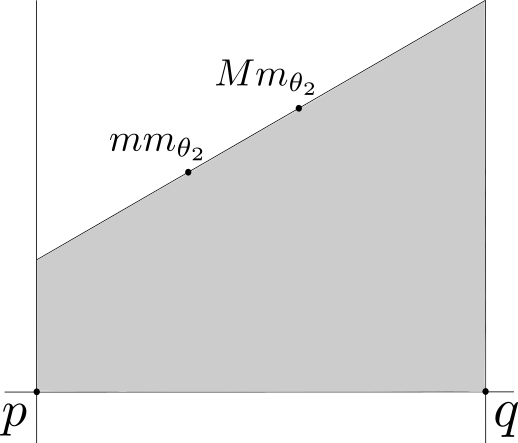}\\
\textbf{Figure 16:} $T_{\theta_2}$
\end{center}

Therefore, we must have $\Gamma_\frac{\pi}{2}^{-1}(mm_{\theta_2})\leq\Gamma_\frac{\pi}{2}^{-1}(Mm_{\theta_2})$.\\
\\
Similarly, we have $\Gamma_\frac{\pi}{2}^{-1}(mm_{\theta_1})\leq\Gamma_\frac{\pi}{2}^{-1}(Mm_{\theta_1})$, and we also have $\Gamma_\frac{\pi}{2}^{-1}(Mm_{\theta_2})\leq\Gamma_\frac{\pi}{2}^{-1}(mm_{\theta_1})$ by considering the diagram below.

\begin{center}
\includegraphics[width=0.65\textwidth]{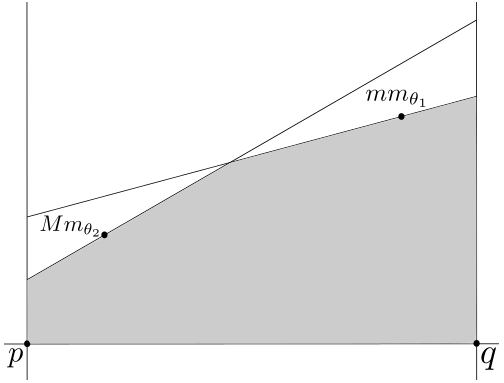}\\
\textbf{Figure 17:} $mm_{\theta_1}$ and $Mm_{\theta_2}$
\end{center}

\end{proof}

\begin{Claim} \label{above line}
All points in $\{(x,y)\in \mathcal{M}_{\theta+\frac{\pi}{2}} \mid \frac{\epsilon_l}{4}\leq x\leq w-\frac{\epsilon_r}{4}\}$ lie strictly above the line $y=x\tan\theta$.
\end{Claim}

\begin{proof}
Let $(x,y)\in\{(x,y)\in \mathcal{M}_{\theta+\frac{\pi}{2}} \mid \frac{\epsilon_l}{4}\leq x\leq w-\frac{\epsilon_r}{4}\}$. We want to show that $y>x\tan\theta$. Since $(x,y)\in \mathcal{M}_{\theta+\frac{\pi}{2}}$, we have $x=\frac{x_1+x_2}{2}$ and $y=\frac{y_1+y_2}{2}$ for some $(x_1,y_1)\in im(\gamma_{\theta+\frac{\pi}{2}})$ and $(x_2,y_2)\in im(\Gamma_{\theta+\frac{\pi}{2}})$, with $(x_1,y_1),(x,y),(x_2,y_2)$ lying on a line of angle $\theta+\frac{\pi}{2}$. Then we have

$$|x_2-x|=\dfrac{|y_2-y|}{|\tan(\theta+\frac{\pi}{2})|}=|y_2-y_0|\tan\theta$$

By the definition of $h$, we must have $|y_2-y|\leq h$. So we have

$$|x_2-x|\leq h\tan\theta<h\tan\epsilon<min(\dfrac{\epsilon_l}{8},\dfrac{\epsilon_r}{8})$$

(the last inequality comes from the definition of $\epsilon$)\\
\\
Since $\frac{\epsilon_l}{4}\leq x\leq w-\frac{\epsilon_r}{4}$ and $|x_2-x|<min(\dfrac{\epsilon_l}{8},\dfrac{\epsilon_r}{8})$, we must have $\frac{\epsilon_l}{8}\leq x_2\leq w-\frac{\epsilon_r}{8}$. Hence, by the definition of $\epsilon_y$, we must have $y_2\geq\epsilon_y$. So we have

$$y=\frac{y_1+y_2}{2}\geq\frac{y_2}{2}\geq\frac{\epsilon_y}{2}>w\tan\epsilon\geq x\tan\theta$$
\end{proof}

Intuitively, Claim \ref{above line} states that the pseudopath $\mathcal{M}_{\theta+\frac{\pi}{2}}$ has to pass through the region $R$. We make it precise in the next claim.\\
\\
Let $Z_{SW}:=(\frac{\epsilon_l}{4},\frac{\epsilon_l}{4}\tan\theta)$, $Z_{SE}:=(w-\frac{\epsilon_l}{4},(w-\frac{\epsilon_l}{4})\tan\theta)$ be the two bottom corners of $R$. Let $Z_W:=\{(\frac{\epsilon_l}{4},y) \mid m_\theta\leq\frac{\epsilon_l}{4}\sin\theta-y\cos\theta\leq 0\}$ be the left vertical edge of $R$, and let $Z_E:=\{(w-\frac{\epsilon_r}{4},y) \mid m_\theta\leq(w-\frac{\epsilon_r}{4})\sin\theta-y\cos\theta\leq 0\}$ be the right vertical edge of $R$.

\begin{center}
\includegraphics[width=0.65\textwidth]{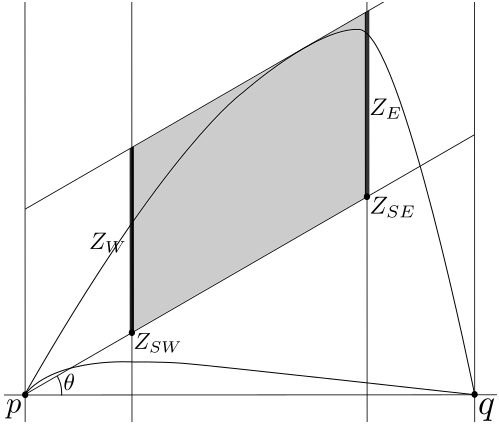}\\
\textbf{Figure 18:} The shaded region is $R$. Showing $Z_{SW}$, $Z_{SE}$, $Z_W$, $Z_E$. 
\end{center}

Let $\mathcal{M}_{\theta+\frac{\pi}{2}}^R:=Z_W\cup Z_E\cup\{(x,y)\in \mathcal{M}_{\theta+\frac{\pi}{2}} \mid \frac{\epsilon_l}{4}\leq x\leq w-\frac{\epsilon_r}{4}\}$. By Claim \ref{above line}, we know that $\mathcal{M}_{\theta+\frac{\pi}{2}}^R\subseteq R$.

\begin{Claim} \label{horizontal pseudopath}
In $R$, $\mathcal{M}_{\theta+\frac{\pi}{2}}^R$ is a pseudopath between $Z_{SW}$ and $Z_{SE}$.
\end{Claim}

\begin{proof}
Let $U$ be an open set in $R$ and $\mathcal{M}_{\theta+\frac{\pi}{2}}^R\subseteq U$. We want to show that there exists a path inside $U$ going from $Z_{SW}$ to $Z_{SE}$. Since $U$ is open in $R$, $U=U^{'}\cap R$ for some $U^{'}$ open in $Rec_\theta$.\\
\\
Let $U^{''}:=U^{'}\backslash\{(x,y) \mid \frac{\epsilon_l}{4}\leq x\leq w-\frac{\epsilon_r}{4} \text{ and }y\leq x\tan\theta\}$. Then $U^{''}$ is still open, and $(\mathcal{M}_{\theta+\frac{\pi}{2}}^R\backslash\{Z_{SW},Z_{SE}\})\subseteq U^{''}$. By Claim \ref{above line}, the set $\{(x,y) \mid \frac{\epsilon_l}{4}\leq x\leq w-\frac{\epsilon_r}{4} \text{ and }y\leq x\tan\theta\}$ we removed does not contain any points from $\mathcal{M}_{\theta+\frac{\pi}{2}}$.\\
\\
Let $U^{'''}:=U^{''}\cup\{(x,y)\in Rec_\theta \mid x<\frac{\epsilon_l}{4} \text{ or } w-\frac{\epsilon_r}{4}<x\}$. Then $U^{'''}$ is an open set, and we have $\mathcal{M}_{\theta+\frac{\pi}{2}}\subseteq U^{'''}$. Also, by construction, we have $(U\backslash\{Z_{SW},Z_{SE}\})\subseteq U^{'''}$. By Claim \ref{median is pseudopath}, there exists a path $\alpha:[0,1]\rightarrow U^{'''}$ going from $p$ to $q$. Since $U^{'''}$ does not contain any points from $\{(x,y) \mid \frac{\epsilon_l}{4}\leq x\leq w-\frac{\epsilon_r}{4} \text{ and }y\leq x\tan\theta\}$, $\alpha$ must intersect $Z_W$ and $Z_E$.\\
\\
Let $a_0:=max\{t\in[0,1] \mid \alpha(t)\in Z_W\}$ and $a_1:=min\{t\in[a_0,1] \mid \alpha(t)\in Z_E\}$. Then $\alpha|_{[a_0,a_1]}$ is a path in $R$ going from $\alpha(a_0)$ to $\alpha(a_1)$. Concatenating $\alpha|_{[a_0,a_1]}$ with a path inside $Z_W$ going from $Z_{SW}$ to $\alpha(a_0)$ and a path inside $Z_E$ going from $\alpha(a_1)$ to $Z_{SE}$, we get a path inside $U$ going from $Z_{SW}$ to $Z_{SE}$.

\begin{center}
\includegraphics[width=0.5\textwidth]{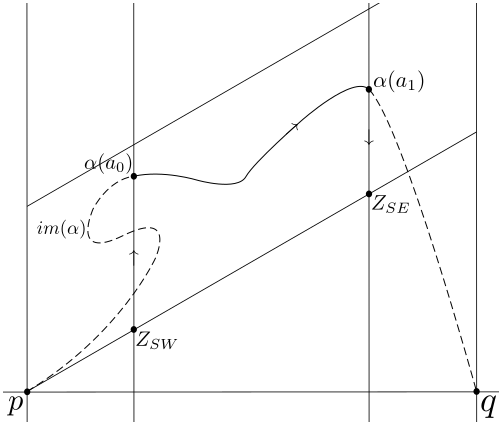}\\
\textbf{Figure 19:} Construction of a path inside $U$ going from $Z_{SW}$ to $Z_{SE}$
\end{center}

\end{proof}

We have finished proving Claim \ref{horizontal pseudopath}, which is an analogue of Claim \ref{median is pseudopath} for $\mathcal{M}_{\theta+\frac{\pi}{2}}$ in $R$. Now we work on an analogue of Claim \ref{median is pseudopath} for $\mathcal{M}_\theta$ in $R$.\\
\\
Let $L$ be the line segment $\{(x,y) \mid 0\leq x\leq w \text{ and }y=x\tan\theta\}$.\\
Let $L_l$ be the line segment $\{(x,y)\in L \mid x\leq\frac{\epsilon_l}{8}\}$, and $L_r$ be the line segment $\{(x,y)\in L \mid w-\frac{\epsilon_r}{8}\leq x\}$.\\
\\
Consider the set of points $S:=\{(x,y)\in im(\Gamma_\frac{\pi}{2}) \mid \frac{\epsilon_l}{8}\leq x\leq w-\frac{\epsilon_r}{8}\}$. By the definition of $\epsilon_y$, whenever $(x,y)\in S$, we have must have

$$y\geq\epsilon_y>2w\tan\epsilon>w\tan\epsilon>x\tan\theta$$

So all points in $S$ lie strictly above L.\\
\\
Let $t_l:=max\{t\in[0,1] \mid \Gamma_\frac{\pi}{2}(t)\in L_l\}$. Note that $t_l$ exists because of compactness and the fact that $p\in L_l$. Let $t_r:=min\{t\in[t_l,1] \mid \Gamma_\frac{\pi}{2}(t)\in L_r\}$. Note that $t_r$ exists because $q$ is below $L$ and all points in $S$ are above $L$. By construction, all points in $im(\Gamma_\frac{\pi}{2}|_{(t_l,t_r)})$ lie strictly above $L$.\\
\\
Let $\overline{\gamma}$ be a Jordan curve formed by concatenating $\Gamma_\frac{\pi}{2}|_{[t_l,t_r]}$ with a path in $L$ going from $\Gamma_\frac{\pi}{2}(t_r)$ to $\Gamma_\frac{\pi}{2}(t_l)$.

\begin{center}
\includegraphics[width=0.5\textwidth]{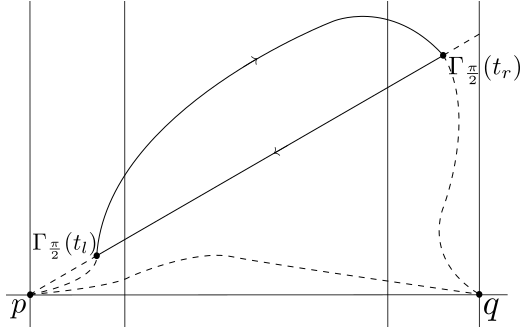}\\
\textbf{Figure 20:} Construction of $\overline{\gamma}$
\end{center}

For every mathematical object we have constructed from $\gamma$, we add a bar on top to denote the same object constructed from $\overline{\gamma}$ instead of $\gamma$.\\
\\
Note that $\overline{M_\theta}=0$ because $im(\overline{\gamma})$ does not go below $L$. So, $\overline{mM_\theta}=\Gamma_\frac{\pi}{2}(t_l)$ and $\overline{MM_\theta}=\Gamma_\frac{\pi}{2}(t_r)$. Now we consider $\overline{mm_\theta}$ and $\overline{Mm_\theta}$. Since the $x$-coordinate of $\Gamma_\frac{\pi}{2}(t_l)$ is less than $\epsilon_l$ and the $x$-coordinate of $\Gamma_\frac{\pi}{2}(t_r)$ is larger than $\epsilon_r$, by the definition of $\epsilon_l$ and $\epsilon_r$, we must have $t_l<\Gamma_\frac{\pi}{2}^{-1}(Mm_{\epsilon_A})$ and $\Gamma_\frac{\pi}{2}^{-1}(mm_0)<t_r$.\\
\\
By Claim \ref{path order}, we have
$$\Gamma_\frac{\pi}{2}^{-1}(Mm_{\epsilon_A})\leq\Gamma_\frac{\pi}{2}^{-1}(mm_\theta)\leq\Gamma_\frac{\pi}{2}^{-1}(Mm_\theta)\leq\Gamma_\frac{\pi}{2}^{-1}(mm_0)$$
Hence, we have $t_l<\Gamma_\frac{\pi}{2}^{-1}(mm_\theta)\leq\Gamma_\frac{\pi}{2}^{-1}(Mm_\theta)<t_r$. So, $mm_\theta$ and $Mm_\theta$ are in $im(\overline{\gamma})$, and hence $\overline{mm_\theta}=mm_\theta$ and $\overline{Mm_\theta}=Mm_\theta$. Therefore, $\overline{A_\theta}=A_\theta$.\\
\\
Let $Z_S:=\overline{B_\theta}=$mid-point of $\Gamma_\frac{\pi}{2}(t_l)$ and $\Gamma_\frac{\pi}{2}(t_r)$, and let $Z_N:=\overline{A_\theta}=A_\theta$. By Claim \ref{median is pseudopath}, in $\overline{Rec_\theta}$, $\overline{\mathcal{M}_\theta}$ is a pseudopath between $Z_S$ and $Z_N$.

\begin{center}
\includegraphics[width=0.5\textwidth]{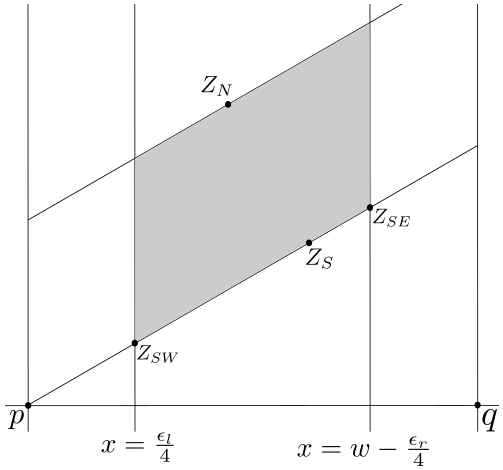}\\
\textbf{Figure 21:} The shaded region is $R$. Showing $Z_{SW}$, $Z_{SE}$, $Z_N$, $Z_S$.
\end{center}

\begin{Claim} \label{vertical pseudopath}
All points in $\overline{\mathcal{M}_\theta}$ have $x$-coordinate strictly between $\frac{\epsilon_l}{4}$ and $w-\frac{\epsilon_r}{4}$.
\end{Claim}

\begin{proof}
Recall that by Claim \ref{path order}, we have

$$\Gamma_\frac{\pi}{2}^{-1}(Mm_{\epsilon_A})\leq\Gamma_\frac{\pi}{2}^{-1}(mm_\theta)\leq\Gamma_\frac{\pi}{2}^{-1}(Mm_\theta)\leq\Gamma_\frac{\pi}{2}^{-1}(mm_0)$$

So, we have

$$im(\overline{\gamma_\theta})=im(\Gamma_\frac{\pi}{2}|_{[t_l,\Gamma_\frac{\pi}{2}^{-1}(\overline{mm_\theta})]})=im(\Gamma_\frac{\pi}{2}|_{[t_l,\Gamma_\frac{\pi}{2}^{-1}(mm_\theta)]})\subseteq im(\Gamma_\frac{\pi}{2}|_{[0,\Gamma_\frac{\pi}{2}^{-1}(mm_0)]})$$

So, using the definition of $\epsilon_r$, we know that all points in $im(\overline{\gamma_\theta})$ have $x$-coordinate  less than or equal to $w-\epsilon_r$. Hence, by the definition of $\overline{\mathcal{M}_\theta}$, all points in $\overline{\mathcal{M}_\theta}$ have $x$-coordinate less than or equal to $\frac{(w-\epsilon_r)+w}{2}=w-\frac{\epsilon_r}{2}<w-\frac{\epsilon_r}{4}$.\\
\\
Similarly, 
$$im(\overline{\Gamma_\theta})=im(\Gamma_\frac{\pi}{2}|_{[\Gamma_\frac{\pi}{2}^{-1}([\overline{Mm_\theta}),t_r]})=im(\Gamma_\frac{\pi}{2}|_{[\Gamma_\frac{\pi}{2}^{-1}([Mm_\theta),t_r]})\subseteq im(\Gamma_\frac{\pi}{2}|_{[\Gamma_\frac{\pi}{2}^{-1}(Mm_{\epsilon_A}),1]})$$
Hence, all points in $im(\overline{\Gamma_\theta})$ have $x$-coordinate greater than or equal to $\epsilon_l$, and hence all points in $\overline{\mathcal{M}_\theta}$ have $x$-coordinate greater than or equal to $\frac{\epsilon_l+0}{2}=\frac{\epsilon_l}{2}>\frac{\epsilon_l}{4}$.
\end{proof}

We have proved Claim \ref{vertical pseudopath}. We now complete the remaining of the proof of Proposition \ref{2 corner}.\\
\\
Lemma \ref{median is pseudopath} and Claim \ref{vertical pseudopath} imply that in $R$, $\overline{\mathcal{M}_\theta}$ is a pseudopath between $Z_N$ and $Z_S$. Together with Claim \ref{horizontal pseudopath} and Lemma \ref{pseudopaths intersect}, we know that $\overline{\mathcal{M}_\theta}$ and $\mathcal{M}_{\theta+\frac{\pi}{2}}^R$ must intersect inside $R$.\\
\\
By Claim \ref{vertical pseudopath}, $\overline{\mathcal{M}_\theta}$ does not touch $Z_W$ nor $Z_E$. Hence, $\overline{\mathcal{M}_\theta}$ and $\mathcal{M}_{\theta+\frac{\pi}{2}}$ must intersect inside $R$. By Claim \ref{above line}, $\mathcal{M}_{\theta+\frac{\pi}{2}}$ does not touch the bottom edge of $R$. Note that $\mathcal{M}_{\theta+\frac{\pi}{2}}$ also does not touch the top edge of $R$ because each point in $\mathcal{M}_{\theta+\frac{\pi}{2}}$ is the mid-point of a point in $\Gamma_{\theta+\frac{\pi}{2}}$ (which cannot be above the top edge of $R$ because of the definition of $m_\theta$) and a point in $\gamma_{\theta+\frac{\pi}{2}}$ (which is strictly below the top edge of $R$). Hence, $\overline{\mathcal{M}_\theta}$ and $\mathcal{M}_{\theta+\frac{\pi}{2}}$ must intersect inside the interior of $R$. Let $Z$ be such an intersection.\\
\\
The point $Z$ corresponds to a rhombus with diagonals being a line of angle $\theta$ and a line of angle $\theta+\frac{\pi}{2}$. The two vertices that correspond to the diagonal of angle $\theta$ lie in $im(\overline{\gamma})$, and they are distinct because $Z$ does not lie in the top edge of $R$. In fact, those two vertices lie in $im(\gamma)$ because $Z$ does not lie in $L$. The two vertices that correspond to the diagonal of angle $\theta+\frac{\pi}{2}$ lie in $im(\gamma)$, and they are distinct because $p,q\notin R$.\\
\\
Hence, $Z$ corresponds to an inscribed rhombus of $\gamma$ of angle $\theta$.

\section{Discussion}
We offer some speculation on how one may use the ideas of this paper towards the original inscribed square problem. One possibility would be to prove that an inscribed rhombus of angle $\theta$ always exist even for those $\theta$ having special corners and not covered in Proposition \ref{2 corner}. The existence of such a rhombus seems intuitive by looking at how the pseudopaths $\mathcal{M}_\theta$ and $\mathcal{M}_{\theta+\frac{\pi}{2}}$ behave near the special points when $\gamma$ is nice enough, or by intuitively thinking about ``areas enclosed'' by those pseudopaths when they don't form ``loops''. Note that we have not defined those concepts in quotation marks for pseudopaths. Intuition comes from treating them as paths.\\
\\
At the same time, we may try to prove some continuity statements on those rhombi when $\theta$ varies. If both are being proven, then we may try to apply some intermediate value theorem arguments like what Arnold Emch did with nice analytic Jordan curves \cite{emch}.

\bibliographystyle{amsplain}
\bibliography{rhombi_proof}

\end{document}